\documentclass[11pt]{amsart} 

\usepackage{color}
\usepackage{amsthm}
\usepackage{graphicx}
\usepackage{amssymb}
\usepackage[mathscr]{eucal}
\usepackage{amsfonts}
\usepackage{amsmath}
\usepackage{amsthm}
\usepackage{amssymb}
\usepackage{latexsym}

\newcommand{\bc}{\begin{center}}
\newcommand{\ec}{\end{center}}

\newcommand{\be}{\begin{equation}}
\newcommand{\ee}{\end{equation}}

\newcommand{\beqn}{\begin{eqnarray*}}
\newcommand{\eeqn}{\end{eqnarray*}}

\newtheorem{theorem}{Theorem}
\newtheorem*{theorem*}{Theorem}

\theoremstyle{definition}
\newtheorem{definition}{Definition}
\newtheorem{remark}{Remark}

\usepackage{pinlabel}
\newcommand{\pic}[2]{\raisebox{-.5\height}{\includegraphics[scale=#2]{#1}}}
\newcommand{\pica}[2]{\raisebox{-.57\height}{\includegraphics[scale=#2]{#1}}}
\newcommand\Xor{\pic{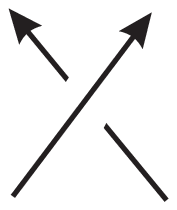}{.50}}
\newcommand\Yor{\pic{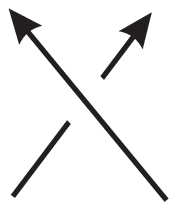} {.50}}
\def\Idor{\pic{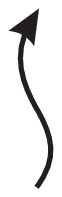} {.50}}
\def\rcurlor{\pic{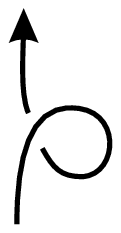} {.50}}
\def\lcurlor{\pic{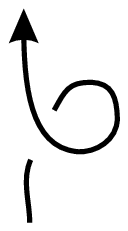} {.50}}
\newcommand\unknot{\pic{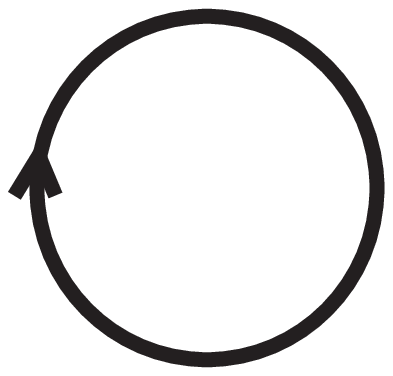} {.20}}

\newcommand\xiibraidAnn{\pica{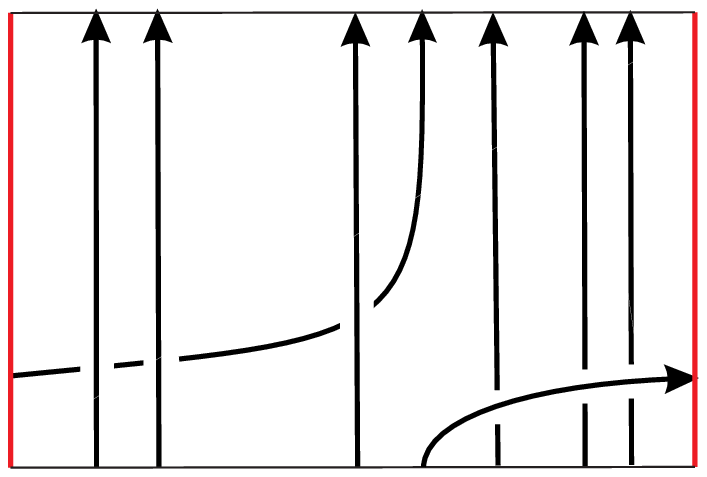}{.36}}
\newcommand\taubraidAnn{\pic{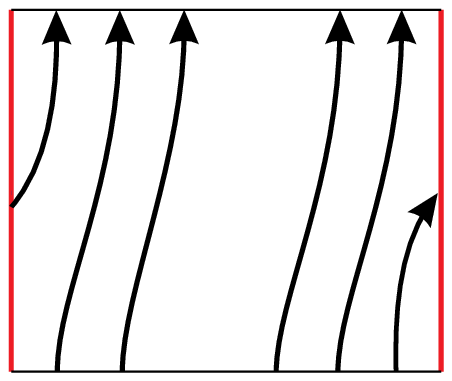}{.45}}
\newcommand\closeAnn{\pic{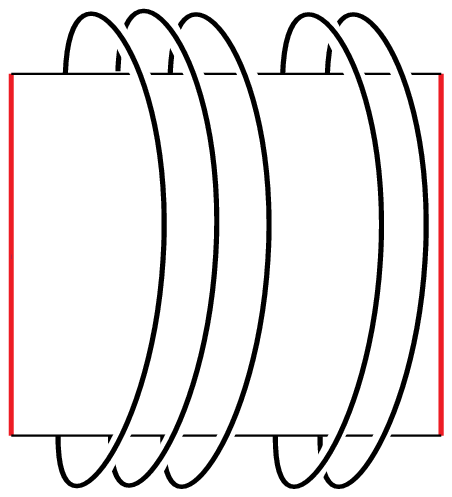}{.45}}
\newcommand\backid{\pic{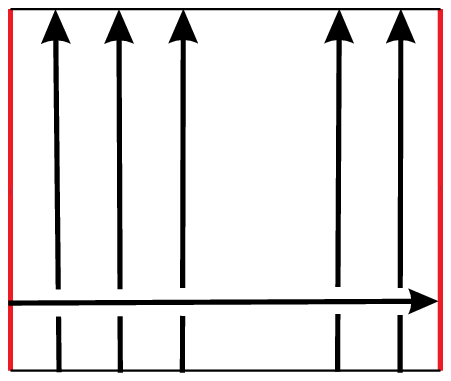}{.45}}
 \newcommand\frontid{\pic{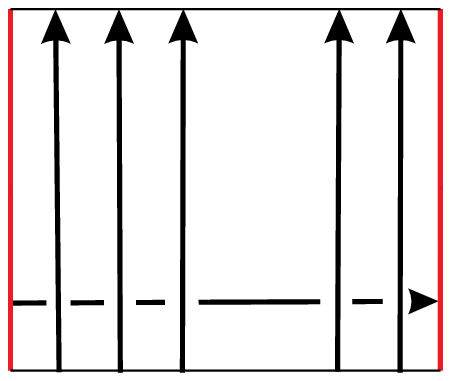}{.45}}
\newcommand{\xx}{{\mathbf{x}}}

\def\Idor{\pic{idor} {.50}}
\def\Idorband{\pic{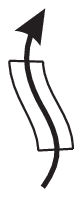} {.60}}
\def\Idorrightband{\pic{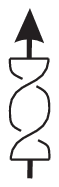} {.60}}
\def\Idorleftband{\pic{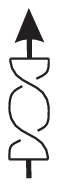} {.60}}
\def\rcurlor{\pic{rcurlor.eps} {.50}}
\def\lcurlor{\pic{lcurlor.eps} {.50}}
\def\rcurlorband{\pic{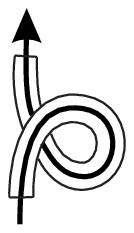} {.60}}
\def\lcurlorband{\pic{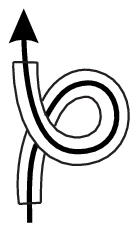} {.60}}
\def\unknotorband{\pic{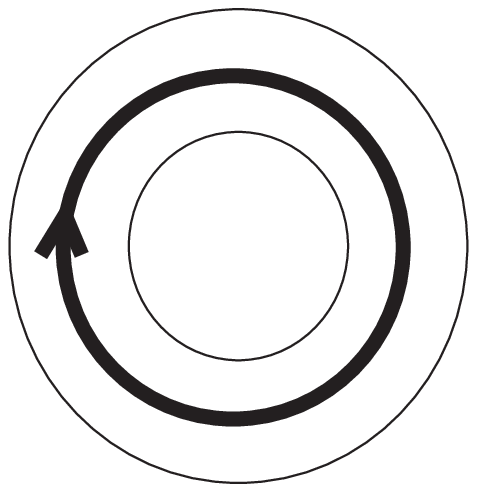} {.150}}

 \newcommand\Xorband{\pic{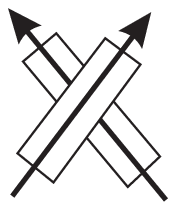}{.50}}
\newcommand\Yorband{\pic{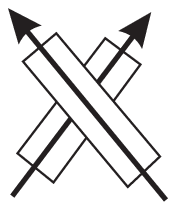} {.50}}
\newcommand\Iorband{\pic{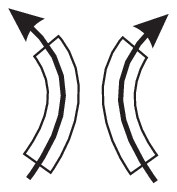} {.50}}

\def\CC{\mathcal{C}}
\def\Cl{\mathrm{Cl}}
\def\Z{\mathbb{Z}}
\def\sk{\mathrm{Sk}}
\newcommand {\bsk}{\mathrm{BSk}}
\def\adj{\mathrm{adj}}
\def\x{\times}

\begin{document}
\title{Tangles in affine Hecke algebras}
\author{Hugh Morton}
\date{\today}
\address{Department of Mathematical Sciences, University of Liverpool, Peach Street, Liverpool L69 7ZL, UK}
\email{morton@liverpool.ac.uk}

\begin{abstract}
The affine Hecke algebra $\dot H_n$ of type $A$ is often presented as a quotient of the braid algebra of $n$-braids in the annulus. This leads to diagrammatic representations in terms of braids in the annulus, subject to a quadratic relation for the simple Artin braids, as in the description by Graham and Lehrer in \cite{GL03}.

I show here that the use of more general framed oriented $n$-tangle diagrams in the annulus, subject to the Homfly skein relations, produces an algebra which is isomorphic to $\dot H_n$ with an extended ring of coefficients. This setting allows the use of some attractive diagrams for elements of $\dot H_n$, using closed curves as well as braids, and gives neat pictures for its central elements.
\end{abstract}
\maketitle
\section{Introduction}
In the course of looking for Homfly-based models for double affine Hecke algebras of type $A$, and other related algebras Samuelson and I have made use of  Homfly\footnote{I should really use the extended acronym HOMFLYPT, particularly since  Traczyk is  my co-author on one of the references. The abbreviation Homfly has, however,  the merit of reading a little bit more easily.}  skeins of braids and more general tangles in the annulus and torus \cite{MS19}. Although we had been familiar with the braid pictures of Graham and Lehrer \cite{GL03} we were not able to find an explicit reference in print to the anticipated extension for the annulus to include the use of closed curves, when asked about this recently by Alistair Savage. 

I have written this paper partly in response to Savage's question, to give a good account of the exact relations between the affine Hecke algebras and their full skein representations.
Following a similar analysis by myself and Traczyk \cite{MT90} of the finite Hecke algebras $H_n$ in terms of oriented $n$-tangles in $D^2$ I use  the same careful choice of ring $\Lambda$ for the Homfly skeins. This can be coupled with Turaev's result \cite{Tur88} that the Homfly skein $\CC$ of closed curves in the annulus is a free polynomial algebra to prove that the Homfly skein $\sk_n(A)$ of $n$-tangles in the annulus $A$  has the following algebra structure.
\begin{theorem*}[Theorem \ref{annulusHecke}]
\[\sk_n(A)\cong \dot H_n\otimes \CC\]
\end{theorem*}

Rather than incorporate the proof of theorem \ref{annulusHecke} into our work on the double affine Hecke algebras \cite{MS19} it seemed preferable to present it as a self-contained note, more along the lines of the old work in \cite{MT90}, which itself does not seem to be widely available.

\section{Background}

We start with a brief account of Homfly skein theory, directed at the construction of algebras related to a surface $F$, so as to establish the exact parameters we need, before concentrating on the annulus, and briefly reviewing \cite{MT90}.

For a $3$-manifold $M$ the Homfly {skein} $\sk(M)$ is based on oriented framed curves in $M$ up to isotopy.
\begin{definition}
The (Homfly) skein $\sk(M)$ of a $3$-manifold $M$ is the set of 
 $\Lambda$-linear combinations of framed oriented curves in $M$ up to isotopy, factored out by three
local linear relations.

These are 

\begin{enumerate}
\item $ \quad\Xorband-\Yorband \ =\ z\ \Iorband $ \quad{ (quadratic),}
\item $\quad\Idorleftband\ =\ \lcurlorband\ =\ v\ \Idorband, \quad\Idorrightband\ =\ \rcurlorband \ =\ v^{-1}\ \Idorband $\quad{\rm (framing),} 
\item $\quad\Idorband \unknotorband \ =\  \delta\  \Idorband $ \quad{\rm (unknot).}
\end{enumerate}
\end{definition}

In these relations the local framing is drawn as a band, and we assume that the diagrams are identical outside the part in a $3$-ball as shown.

For compatibility we need $\delta z=v^{-1}-v$  in $\Lambda$.  We shall work here with
\[\Lambda=\Z[z,\delta,v^{\pm1}]/<\delta z=v^{-1}-v> \subset \Z[z^{\pm1}, v^{\pm1}].\]

This is in many ways the natural ring to use for the basic Homfly polynomials, since the only occurrences of $z^{-1}$ which are needed arise from factors of $\delta$.  There is an advantage in using this ring, since it is then possible to set $z=0$ and $v=1$ while retaining $\delta$. This gives us a ring homomorphism $e:\Lambda \to \Z[\delta]$ with $e(z)=0, e(v)=1, e(\delta)=\delta$.

\section{Surface skeins}

  For a surface $F$ the skein of the thickened surface $M=F\times I$ admits a natural product making it an algebra over $\Lambda$. The product structure is defined by stacking copies of $F\x I$.
 Write $\sk(F)$ rather than $\sk(F\x I)$ for this algebra.

In the  case $F=D^2$ we have $\sk(F)\cong\Lambda$, with the empty diagram acting as the identity element and a link $L$ in $\sk(D^2)$ corresponding to a normalised version of its Homfly polynomial in $\Lambda$.

Where $F$ is the annulus $A=S^1\x I$ the algebra $\sk(A):=\CC$ is commutative. It was shown by Turaev \cite{Tur88} that $\CC$ is a free polynomial algebra on commuting  generators $\{A_m, m\ne 0\in\Z\}$. Subsequent work has produced more interesting interpretations of $\CC$ and useful bases  besides the basis of monomials in Turaev's generators. For example a basis $\{Q_{\lambda,\mu} \}$ indexed by pairs of partitions of integers is described in \cite{HM06}, and other features of $\CC$ are discussed in \cite{MM08}.

For other surfaces $F$ the algebra $\sk(F)$ is generally non-commutative. A presentation of the algebra $\sk(T^2)$  of the torus is given in \cite{MS17} with generators $\{P_\xx, \xx \ne 0\in \Z^2\}$, while  the sets of generators get uncomfortably large once the surface is more complicated.

\section{Relative skein algebras}

We can extend our skeins and the resulting algebras by including framed arcs as well as closed curves.
Fix $n$ points $J\subset F$ and include $n$ oriented framed arcs from  $J\x \{0\}$ to $J\x\{1\}$,  along with oriented framed curves to give an algebra for each $n$ on applying the skein relations. We assume that there is a  framing chosen at each point in $J$, which is fixed under any isotopy of the framed arcs. 

Write $\sk_n(F)$ for the resulting  algebra with $\sk_0(F)$  written as $\sk(F)$ for the skein of closed curves in $F$. 
\begin{remark} Up to isomorphism the algebra depends only on $n=|J|$, not on the exact choice of $J$ or the framing at the points of $J$.
\end{remark}

In $\sk_n(F)$ the $n$-braids play an important role - these are tangles with  $n$ monotonic arcs and no closed curves. If we stick to these, and just impose the quadratic relations then the known presentations of surface braid groups with the addition of the quadratic relation in the form $a^2-1=az$ for $a=\sigma_i$ give a presentation of the resulting algebra $\bsk_n(F)$. 

The simplest case is $F=D^2$, where $\bsk_n(D^2)$ gives the Hecke algebra $H_n(z)$. This has a nice basis of $n!$ elements, represented by the positive permutation braids $\{b_\omega, \omega\in S_n\}$.

Traczyk and I noted  \cite{MT90} that the full skein $\sk_n(D^2)$ is also spanned by these elements, and used the Homfly polynomial with the simplification given by mapping to $\Z[\delta]$ to prove them independent in $\sk_n(D^2)$. 
\begin{theorem} [Morton-Traczyk, 1986]
 \[\sk_n(D^2)\cong H_n(z)\otimes \Lambda\] \end{theorem}

\section{The skein algebras of the annulus}

We will follow the pictorial approach of \cite{GL03} to describe the skein $\sk_n(A)$. Draw diagrams of braids and tangles on a square with the left and right hand sides identified. The $n$ input points at the bottom  are joined by (framed) arcs to $n$ output points at the top. We can allow additional  oriented closed curves in the diagram while we don't assume that the arcs rise monotonically. We  use the blackboard  convention to indicate the framing when drawing diagrams in this way.

Composition is given by placing one diagram on top of the other.

The $n$-braids in $A$ are generated by the regular Artin braids $\{\sigma_i\}$, which stay within the square, and a further braid $\tau$ in which all strings move one place to the right as shown here, 
\[\tau = \taubraidAnn.\]
This braid satisfies $\sigma_i \tau=\tau\sigma_{i+1}$ for $i=1,\ldots, n-2$ and $\sigma_{n-1}\tau^2=\tau^2\sigma_1$.
\begin{remark} Graham and Lehrer introduce $\sigma_n:=\tau^{-1}\sigma_{n-1}\tau$ to work uniformly with the indices $\bmod \ n$.
\end{remark}

 We write $x_i$ for the braid \[ x_i\ =\  \xiibraidAnn\] in which the $i$th point moves once round the annulus as shown, and no others move.
The elements $x_i$ commute among themselves and can be used, along with the positive permutation braids, to  give a linear basis for the affine Hecke algebra $\dot H_n$.  

Explicitly, the elements
\[\{v_\rho:=x_1^{r_1}x_2^{r_2}\ldots x_n^{r_n} b_\omega\}, \rho=(\mathbf r_\rho,\omega_\rho)\in I=\Z^n\times S_n\] form a linear basis for the affine Hecke algebra $\dot H_n(z)$ over the polynomial  ring $\Z[z]$. Here $\dot H_n(z)$ is defined as the quotient of the group ring of $n$-braids in the annulus, factored by the quadratic relation
\[a^2-1=za\] for $a=\sigma_i$. The relation may also be put in factorised form \[(a-s)(a+s^{-1})=0, z=s-s^{-1}
.\]

\begin{remark}
The positive permutation braid $b_\omega$ for the permutation $\omega\in S_n$ can be defined geometrically as a braid having a diagram with $n$  arcs only such that the string $i$, which starts from point $i$ at the bottom, goes to point $\omega(i)$ at the top, and at any crossing point string $i$ crosses over string $j$ at any crossing point when $i<j$. It is well-known that any braid diagram satisfying this condition   depends only on the permutation $\omega$ of the endpoints.

The braids $v_\rho$ in the annulus satisfy the same condition there, in that  at any crossing string $i$ lies over string $j$ if $i<j$.
\end{remark}

In the skein $\sk_n(A)$ we will have diagrams that include closed curves. In particular we have diagrams such as 
\[\frontid \] where a single closed curve lies behind the identity braid.
More generally we can define an algebra homomorphism 
\[\phi:\CC\to \sk_n(A)\] by setting
\[\phi(X)=\labellist\small
\pinlabel {$X$} at 345 335
\endlabellist\frontid\] where the combination of curves in the annulus representing $X$ are placed behind the identity braid as shown.

We can now present our description of the full skein $\sk_n(A)$ of the annulus as the affine Hecke algebra $\dot H_n$ with coefficients extended via $\phi$ to $\CC$.

\begin{theorem} \label{annulusHecke}
The full skein  $\sk_n(A)$ of framed oriented $n$-tangles  in the annulus satisfies
\[\sk_n(F)\cong \CC\otimes \dot H_n,\] where $\CC =\sk(A)$ is the full skein of closed curves in the annulus, lying as curves completely behind the identity braid.
\end{theorem}

\begin{proof}[First part of the proof]
The first part of the proof is to show that any element of the skein can be written as a linear combination of products $c\beta  $ with $\beta$ an $n$-braid in $A$, and $c\in\CC$. We can be more specific and choose $\beta=x_1^{r_1}x_2^{r_2}\ldots x_n^{r_n} b_\omega$ for some integer powers $r_1, \ldots, r_n$ and positive permutation braid $b_\omega, \omega\in S_n$.

Work inductively on the number $k$ of crossings in the diagram of an $n$-tangle, drawn in the square with left and right sides identified. 

If $k=0$ the diagram is isotopic to some power  $\tau^l$, along with some number $m\ge 0$ of disjoint unknotted curves, giving $\delta^m \tau^l$. This can readily be written in the required form.

Order the starting points of the tangles $1,\ldots,n$ and put the tangle strings in order into totally ascending form. This means that starting from the input point $1$ and  following along the strings in order the crossing points on string $i$ are met as over-crossings  for the first time when working along  string $i$ from start to finish. If a diagram is not in this form we switch the  crossing we first meet which is not an over-crossing, using the quadratic skein relation. The difference is a smoothed diagram, with one less crossing point, which is covered by our induction hypothesis.

 Once a diagram is in totally ascending form the tangle strings will lie in front of any closed curves, which  then determine an element $c\in \CC$. The tangle strings then lie above each other in order, and can be isotoped independently.  The whole tangle can then be isotoped to a braid in the annulus of the form $v_\rho =x_1^{r_1} x_2^{r_2}\ldots x_n^{r_n} b_\omega$. In this isotopy the framing on the individual strings may have extra twists, when compared to the isotopic braid. This can be accounted for by multiplying by a suitable power of $v$.
 \end{proof}
 \begin{proof}[Second part of the proof]
 In theorem \ref{independence} we prove that the braids $\{v_\rho, \rho\in I\}$ are independent over $\CC$. We have just shown that they form a spanning set for $\sk_n(A)$ regarded as an algebra with coefficients in $\CC$. Hence they form a basis over $\CC$ for $\sk_n(A)$, while $\{v_\rho, \rho\in I\}$ form a basis for $\dot H_n$. This proves the result that
 \[\sk_n(A)\cong\CC\otimes \dot H_n.\]
 \end{proof}

To show the independence of the braids $\{v_\rho\}$ in the skein $\sk_n(A)$ we use a closure map from $n$-tangles to curves in the annulus.

The closure map $\Cl:\sk_n(A)\to\sk(A)=\CC$ is a $\Lambda$-linear map defined by closing off the points at the top by curves in front of our diagram to points at the bottom, using the extra curves shown here. \[ \closeAnn\]

This  map satisfies \[\Cl (c T) =c\Cl(T)\] for any $c\in\CC$ since the curves of $c$ lie behind the identity braid and can move apart from the closure curves  in the annulus. The map $\Cl$ then behaves well with respect to our extra coefficients in $\CC$.

\begin{theorem}\label{independence}
The elements $v_\rho, \rho\in I$ are linearly independent over $\CC$.
\end{theorem}
Our proof of theorem \ref{independence} depends on the following result,  which will be proved later.

\begin{theorem} \label{determinant}
Let $J\subset I$ be any finite subset of $I$, and set $k=|J|$. Define a $k\times k$ matrix $C_J$ with entries $C_{\mu,\nu}$ in the commutative ring $\CC$ for each $\mu,\nu\in J$ by taking
\[C_{\mu, \nu}:=\Cl(v_\mu v_\nu^{-1})\in\CC.\]
Then \[ \det(C_J) \ne 0\in\CC.\]

\end{theorem}

\begin{proof}[Proof of theorem \ref{independence}] Suppose that there is a linear relation \[\sum_{\mu\in J} a_\mu v_\mu=0, a_\mu\in\CC,\] for some finite subset $J\subset I$ with $|J|=k$. Construct the matrix $C_J$ as above, and then
construct the $k\times k$ \emph{adjoint matrix} $\adj(C_J)$, also with entries in the commutative ring $\CC$,  which satisfies \[\adj(C_J)C_J =C_J \adj(C_J)=\det(C_J)\times (I_k).\]
 Write $\adj(C_J):=D_J$ with entries $D_{\nu,\rho},\rho\in J$. Then
\[\sum_{\nu\in J} C_{\mu,\nu} D_{\nu,\rho} =\det(C_j) \delta_{\mu,\rho},\] where $\delta_{\mu,\rho}$ is the Kronecker delta, equal to $1$ when $\rho=\mu$ and $0$ otherwise.

For each $\nu\in J$ we have
\[0=\sum_{\mu\in J} a_\mu v_\mu v_\nu^{-1}=\sum_{\mu\in J} a_\mu C_{\mu,\nu}.\] Then for each $\rho\in J$ we get
\[0=\sum_{\mu,\nu\in J} a_\mu C_{\mu,\nu}D_{\nu,\rho} =\det(C_J)a_\rho.\]

Since  $\det(C_J)\ne 0$ and the ring $\CC$ has no zero-divisors  it follows that  the coefficients $a_\rho$ in the relation are all zero.  The elements $v_\rho, \rho\in J$ in $\sk_n(A)$ are thus independent.
\end{proof}

Once we have proved theorem \ref{determinant} we then have completed the proof of our main theorem  \ref{annulusHecke} that
\[\sk_n(A) \cong \CC\otimes \dot H_n.\]

\begin{remark}
The argument above, using only coefficients $a_\mu\in \Z[s^{\pm 1}, v^{\pm 1}]$, confirms directly the result of Graham and Lehrer \cite{GL03} that the braid skein $\bsk_n(A)$ of the annulus $A$ is isomorphic to the affine Hecke algebra $\dot H_n$, with the same basis elements $\{v_\rho\}$. 
\end{remark}

Before proving theorem \ref{determinant} it is worth looking at some useful features of the skein $\CC$.  
Turaev's original analysis of this skein \cite{Tur88} shows that it is a free polynomial algebra with explicit generators $\{A_m, m\in\Z-\{0\}\}$ each represented by a closed curve with winding number $m$ in the annulus. We will suppose that we are using coefficients in $\Lambda$. There is then a ring homomorphism $p_0:\CC\to \Lambda$ which selects the scalar term only of the polynomial.

Write $\CC'$ for the free polynomial algebra over $\Z[\delta]$ on the same generators $A_m$. Recall that there is a ring homomorphism $e:\Lambda \to \Z[\delta]$ defined by $e(z)=0,e(v)=1,e(\delta )= \delta$. Apply $e$ to the coefficients of each monomial to define a homomorphism  $E:\CC\to\CC'$  which satisfies $p_0(E(X))=e(p_0(X))$ for any $X\in\CC$.

At the skein level the map $E$ allows us to   pull  curves through each other, and to change their framing, since the skein relations simplify to
\bc
$\Xor=\Yor, \quad
\rcurlor =\lcurlor =\Idor, \quad\Idor \ \unknot = \delta \ \Idor$
\ec 
after applying $e$ to the coefficients. More formally, if $X$ and $Y$  are two diagrams related by these simplified relations then $E(X)=E(Y)\in\CC'$. 

In the simplified skein $\CC'$ any link diagram $L$ with components $L_1,..,L_r$ can be separated into the product of the individual components, and each of these components can be altered to one of the curves $A_m$ or to the unknotted curve $\delta:=A_0$. Write $m_i$ for the winding number of $L_i$ in the annulus. We then have the skein representation of $L$ in $\CC'$ as
\[E(L)=\prod_{i=1}^r A_{m_i}.\]

\begin{remark}
Back in the skein $\CC$ we could follow through the steps in the skein resolution of $L$ to write
\[L=v^\tau \prod_{i=1}^r A_{m_i} + K\] for some $\tau\in\Z$ and $K\in (s-s^{-1})\CC$.
\end{remark}

\begin{proof} [Proof of theorem \ref{determinant}]

To show that $\det(C_J)\ne 0\in\CC$ it is enough to prove that $p_0(E(\det(C_J)))\ne 0$  in $\Z[\delta]$. Now the determinant of a matrix is a polynomial in its entries, and both  $ E $ and $p_0$ are ring homomorphisms. Hence \[p_0(E(\det(C_J))) = \det(p_0(E(C_J))). \] We must then calculate  the determinant of the matrix with entries $p_0(E(C_{\mu,\nu}))$.

Now each of the matrix entries $C_{\mu,\nu}\in\CC$ is represented by a single diagram. The closure of the braid $v_\mu v_\nu^{-1}$ is a link with $r\le n$ components $L_1,\ldots,L_r$, where $r$ is the number of disjoint cycles in the permutation $\omega_\mu \omega_\nu^{-1}$. Hence
\[E(C_{\mu,\nu})=\prod_{i=1}^r A_{m_i},\] where $L_i$ has winding number $m_i$ in the annulus and we take $A_0=\delta$.
If any of the winding numbers $m_i$ is non-zero then
\[p_0(E(C_{\mu,\nu}))=0,\] otherwise \[p_0(E(C_{\mu,\nu}))=\delta^r.\]

When $\mu=\nu$ we have $C_{\mu,\mu}=\Cl(Id)=\delta^n$, the unlink with $n$ components. 
If $\omega_\mu\ne\omega_\nu$ then $r<n$.
 If $\omega_\mu=\omega_\nu$ then the braid $v_\mu v_\nu^{-1}$ is pure, so $r=n$. Either $\mu=\nu$ and the entry is $\delta^n$ or  at least one component of the closure of $v_\mu v_\nu^{-1}$ has non-zero winding number in the annulus and the entry is $0$.
 
 Thus the diagonal entries in the matrix $p_0(E (C_J))$ are all $\delta^n$ while the remaining entries are either $0$ or $\delta^r$ for various $r<n$. The element \[p_0(E(\det(C_J)))\in \Z[\delta]\] then has a single leading term $\delta^{kn}$ arising from the diagonal entries in $C_J$. This is not cancelled by the other terms in the  expansion of the determinant, and hence the determinant is non-zero in $\Z[\delta]$. It follows at once that $\det(C_J) \ne 0 \in \CC$. 
\end{proof}

\begin{remark} The proof of independence of the positive permutation braids in the skein $\sk_n(D^2)$ in \cite{MT90} used the evaluation in $\Z[\delta]$ of a similar $k\x k$ matrix, corresponding to the finite subset $J\subset I$ of permutation braids. Indeed our proof here shows the independence of this set of braids. 

 The matrix corresponding to $C_J$ in  \cite{MT90} is not exactly that used here, as it uses mirror images rather than inverses of the spanning elements in its construction, which related to  the geometric construction of a bilinear form on $H_n$. 

The same geometric construction is also used  in \cite{MT90} in a similar argument about the independence of a spanning set for $n$-tangles using the Kauffman $2$-variable skein of $D^2$, aimed at a tangle-based realisation of the Birman-Wenzl-Murakami algebras and their relation with the Brauer algebras. This indeed was the primary goal of \cite{MT90}, and the mirror method was devised because some of the spanning elements used were not invertible in the Kauffman skein.

It is quite likely that an analysis of the Kauffman $2$-variable skein of $n$-tangles in the annulus could be completed along similar lines to the analysis here for the Homfly skein.
\end{remark}
\section{Central elements}
The homomorphism $\phi:\CC\to \sk_n(A)$, placing closed curves behind the identity braid, which we used to introduce scalar coefficients into the algebra has images of the form $\phi(X)$ which are obviously central.  There is a similar homomorphism $\psi:\CC\to\sk_n(A)$ placing the closed curves in front of the identity braid.

The image  \[\psi(Y)=\labellist\small
\pinlabel {$Y$} at 345 335
\endlabellist\backid\]  is again obviously central.

It is interesting to consider how such central elements $\psi(Y)$ look in terms of our algebraic formulation of the skein as $\CC\otimes \dot H_n$. By construction $\psi(X)$ is just the scalar $X\otimes 1$. 

Repeated use of the skein relation shows that 
\[ \backid\ =\ z\sum_i \ \labellist\small
\pinlabel {$i$} at 260 220
\endlabellist \xiibraidAnn\ \ +\  \ \frontid\]

This shows that $\psi(P_1)= z\sum x_i +\phi(P_1)$, where $P_1\in\CC$ is represented by the single core curve of the annulus. In our algebraic setting this reads
\[\psi(P_1)=P_1 +z\sum x_i,\] showing that $\sum x_i$ is central in $\dot H_n$.

There are further elements $P_m\in\CC, m\ne 0\in \Z$, introduced in \cite{Mor02, Mor02b}, with the property that 
\beqn \labellist\small
\pinlabel {$P_m$} at 345 335
\endlabellist\backid \quad - \quad \labellist\small
\pinlabel {$P_m$} at 345 335
\endlabellist\frontid &=&  (s^m-s^{-m})\sum x_i^m. \eeqn

Then $\psi(P_m) = P_m+(s^m-s^{-m})\sum x_i^m$, and so the power sums of $x_1,\ldots, x_n$ and their inverses are also central in $\dot H_n$. We can deduce that all symmetric polynomials in $\{x^{\pm i }\}$ are central in $\dot H_n$. While this is well-known in the algebraic context we get here  a nice geometric realisation.

The elements $\{P_m\}$ can be used as generators of the algebra $\CC$ in place of Turaev's generators. Expanding an element $Y\in\CC$ in terms of $\{P_m\}$ then leads to an expression of the central element $\psi (Y) $ shown above as a symmetric polynomial in $\{x^{\pm i }\}$ with coefficients in $\CC$, for any choice of $Y$.

\bibliography{bibtex_dahamodels}{}
\bibliographystyle{amsalpha}

\end{document}